\documentclass{amsart}
\usepackage{xcolor}
\usepackage{cleveref,amssymb}

\title[Graded polynomial identities of $\mathrm{UT}$]{Graded polynomial identities of the infinite-dimensional upper triangular matrices over an arbitrary field}
\author{Micael Said Garcia}
\address{Department of Mathematics, Instituto de Matem\'atica e Estat\'istica, Universidade de S\~ao Paulo, SP, Brazil}
\email{micael.said@usp.br}
\author{Felipe Yukihide Yasumura}
\address{Department of Mathematics, Instituto de Matem\'atica e Estat\'istica, Universidade de S\~ao Paulo, SP, Brazil}
\email{fyyasumura@ime.usp.br}
\thanks{The first author is supported by CNPq - Brazil. The second named author is supported by Fapesp, grant 2023/03922-8 and 2018/23690-6.}

\newcommand{\F}{\mathbb{F}}

\newtheorem{theorem}{Theorem}[section]
\newtheorem{lemma}[theorem]{Lemma}
\newtheorem{corollary}[theorem]{Corollary}
\newtheorem{Prop}[theorem]{Proposition}
\theoremstyle{definition}
\newtheorem{definition}[theorem]{Definition}
\theoremstyle{remark}
\newtheorem{remark}[theorem]{Remark}

\begin{document}
\begin{abstract}
We compute the graded polynomial identities of the infinite dimensional upper triangular matrix algebra over an arbitrary field. If the grading group is finite, we prove that the set of graded polynomial identities admits a finite basis. We find conditions under which a grading on such an algebra satisfies a nontrivial graded polynomial identity. Finally, we provide examples showing that two nonisomorphic gradings can have the same set of graded polynomial identities.
\end{abstract}
    \maketitle
 
\section{Introduction}
This paper deals with the classification of graded polynomial identities of a non-PI-algebra. Kemer's celebrated result (see \cite{Kemerbook} or \cite{AljKBK2016}) provides a positive answer for the Specht problem in the context of associative algebras over a field of characteristic zero. In the case of graded algebras, the same problem has a positive solution for associative PI-algebras graded by a finite group over a field of characteristic zero (see \cite{AljKB2010}).

We compute the graded polynomial identities of the infinite-dimensional upper triangular matrices. Such algebra is constructed from the direct limit $\mathrm{UT}=\bigcup_{n\in\mathbb{N}}\mathrm{UT}_n$. From construction, we obtain that
$$
\mathrm{Id}(\mathrm{UT})=\bigcap_{n\in\mathbb{N}}\mathrm{Id}(\mathrm{UT}_n)=0.
$$
Hence, $\mathrm{UT}$ is not a PI-algebra. We hope to shed light on the theory of graded polynomial identities of non-PI-algebras, where we cannot derive its Specht property from the results of the paper \cite{AljKB2010}.

The graded polynomial identities of the finite-dimensional upper triangular matrices are described in the papers \cite{VinKoVa2004} (for infinite fields) and \cite{GR2020} (for finite fields). It is worth mentioning that the classification of the group gradings on the same algebra is done in the paper \cite{VaZa2007}. Regarding the nonassociative graded polynomial identities of the finite dimensional upper triangular matrices, including the case where the base field is finite, and related invariants, they were studied in the papers \cite{CMa, CMaS, CorK, DimasSa, GSK, KMa, PdrManu, Y23}.

The classification of graded polynomial identities of an infinite-dimensional algebras (associative or non-associative) is done in several works, for instance, \cite{FDK,FK1,FK2}. However, for the best of our knowledge, few of the works focus on a non-PI algebra.

The paper is structured as follows. In \Cref{sec:preliminaries}, we recall the basic definitions concerning the main objects of study. Then, in \Cref{sec:grpolid}, we classify the graded polynomial identities and find a basis of the relatively free algebra of the graded variety generated by a group grading on $\mathrm{UT}$ (\Cref{mainTheorem}). Moving on to \Cref{sec:conditionsID}, we find conditions for a grading on $\mathrm{UT}$ to satisfy a graded polynomial identity (\Cref{notEventuallyIncomplete} and \Cref{conditionID}). In the next section, we prove the finite basis property if the grading group if finite (\Cref{thm:finitebasis}). Finally, in the last section, we provide some examples to illustrate that non-isomorphic gradings on $\mathrm{UT}$ may satisfy the same set of graded polynomial identities, which is in contrast with the finite-dimensional case.

\section{Preliminaries}\label{sec:preliminaries}
\subsection{Group gradings} Let $G$ be a group and $\mathcal{A}$ an $\mathbb{F}$-algebra. We shall use a multiplicative notation for the group $G$, and denote its neutral element by $1$. A $G$-grading on $\mathcal{A}$ is a vector space decomposition
$$
\mathcal{A}=\bigoplus_{g\in G}\mathcal{A}_g,
$$
such that $\mathcal{A}_g\mathcal{A}_h\subseteq\mathcal{A}_{gh}$, for all $g$, $h\in G$. If a $G$-grading on $\mathcal{A}$ is fixed, then we say that $\mathcal{A}$ is $G$-graded. The component $\mathcal{A}_g$ is called the homogeneous component of degree $g$, and its nonzero elements are said to be homogeneous of degree $g$ as well. Given $x\in\mathcal{A}_g$, $x\ne0$, we denote $\deg_G x=g$. A subspace $\mathcal{S}\subseteq\mathcal{A}$ is called \emph{graded} if $\mathcal{S}=\bigoplus_{g\in G}\mathcal{S}\cap\mathcal{A}_g$. A graded subalgebra (ideal) is a subalgebra (an ideal) that is a graded subspace. The support of the grading is $\mathrm{Supp}\,\mathcal{A}=\{g\in G\mid\mathcal{A}_g\ne0\}$.

If $\mathcal{B}$ is another $G$-graded algebra, then a homomorphism of $G$-graded algebras is a homomorphism of algebras $\varphi:\mathcal{A}\to\mathcal{B}$ such that $\varphi(\mathcal{A}_g)\subseteq\mathcal{B}_g$, for all $g\in G$. A complete reference on the subject of graded algebras is the monograph \cite{EK2013}.

\subsection{Graded polynomial identities} Let $X^G=\bigcup_{g\in G}X^g$, where, for each $g\in G$, we have a set of variables $\{x_1^{(g)},x_2^{(g)},\ldots\}$. The free associative algebra $\mathbb{F}\langle X^G\rangle$ becomes a $G$-graded algebra naturally, imposing $\deg_G x_i^{(g)}=g$. It satisfies the following universal property. For any $G$-graded algebra $\mathcal{A}$ and map $\varphi:X^G\to\mathcal{A}$ respecting degrees (that is, $\varphi(x_i^{(g)})$ is either $0$ or homogeneous of degree $g$), there exists a unique $G$-graded algebra homomorphism $\bar{\varphi}:\mathbb{F}\langle X^G\rangle\to\mathcal{A}$ extending $\varphi$. An element $f=f(x_1^{(g_1)},\ldots,x_m^{(g_m)})\in\mathbb{F}\langle X^G\rangle$ is said to be a graded polynomial identity of $\mathcal{A}$ if $\bar{\varphi}(f)=0$, for all graded algebra homomorphism $\bar{\varphi}:\mathbb{F}\langle X^G\rangle\to\mathcal{A}$. In other words, $f(a_1,\ldots,a_m)=0$, for all $a_1\in\mathcal{A}_{g_1}$, \dots, $a_m\in\mathcal{A}_{g_m}$. The set of all graded polynomial identities of $\mathcal{A}$ is denoted by $\mathrm{Id}_G(\mathcal{A})$. We shall denote the variables of trivial homogeneous degree by $y_1=x_1^{(1)}$, $y_2=x_2^{(1)}$, \dots, and call them even variables. We use the letter $z$ to denote the variables of non-trivial homogeneous degree, and call them odd variables.

The notion of ordinary polynomial identities is recovered if we put $G=1$, the trivial group. We shall denote the set of ordinary polynomial identities of a given algebra $\mathcal{A}$ by $\mathrm{Id}(\mathcal{A})$. It should be noted that, if $G$ is finite, then $\mathrm{Id}(\mathcal{A})\subseteq\mathrm{Id}_G(\mathcal{A})$.

\subsection{Upper triangular matrices}
The algebra of $n\times n$ upper triangular matrices with entries in the field $\mathbb{F}$ is denoted by $\mathrm{UT}_n$. We shall denote by $e_{ij}$ the matrix units, that is, the matrix having an $1$ in the entry $(i,j)$ and $0$ elsewhere. A $G$-grading on $\mathrm{UT}_n$, where $G$ is a group, is said to be \emph{elementary} if there exists a map $\bar{g}:\{1,\ldots,n\}\to G$ such that every $e_{ij}$ is homogeneous and $\deg_G e_{ij}=\bar{g}(i)\bar{g}(j)^{-1}$. In 2007, Valenti and Zaicev (see \cite{VaZa2007}) proved that every group grading on $\mathrm{UT}_n$ is isomorphic to an elementary grading. Moreover, the isomorphism classes of elementary gradings on $\mathrm{UT}_n$ is classified in \cite{VinKoVa2004}. In the same paper, the authors provide a description of the graded polynomial identities satisfied by it, if the base field is infinite. If $\mathbb{F}$ is finite, then the paper \cite{GR2020} gives the computation of its graded polynomial identities.

To state their result, we need a definition. A sequence $\eta=(\eta_1,\ldots,\eta_m)\in G^m$ is said to be $\varepsilon$-good if there exist matrix units $r_1$, \ldots, $r_m\in J(\mathrm{UT}_n)$ such that $r_1\cdots r_m\ne0$ and $\deg_G r_i=\eta_i$. Otherwise, $\eta$ is said to be $\varepsilon$-bad. Now, consider the base field $\mathbb{F}$ endowed with the trivial $G$-grading, and let $\mathcal{S}$ be a set of generators of $\mathrm{Id}_G(\mathbb{F})$. It means that either:
\begin{enumerate}
\item $\mathcal{S}=\{[x_{2i-1}^{(1)},x_{2i}^{(1)}],x_i^{(g)}\mid g\ne1,i\in\mathbb{N}\}$ if $\mathbb{F}$ is infinite, or
\item $\mathcal{S}=\{[x_{2i-1}^{(1)},x_{2i}^{(1)}],(x_i^{(1)})^q-x_i^{(1)},x_i^{(g)}\mid g\ne1,i\in\mathbb{N}\}$ if $\mathbb{F}$ is finite containing $q$ elements.
\end{enumerate}
We can state the description of the graded polynomial identities of $\mathrm{UT}_n$ globally in the following way.

    \begin{theorem}[\cite{VinKoVa2004,GR2020}]
        Let $\varepsilon$ be a $G$-grading on $\mathrm{UT}_n$, over an arbitrary field $\mathbb{F}$. Then $\mathrm{Id}_G(\mathrm{UT}_n,\varepsilon)$ follows from all $f_1\cdots f_r$, where each $f_i\in\mathcal{S}$ (defined above), $(\deg_G f_1,\ldots,\deg_G f_r)$ is $\varepsilon$-bad, and $r\le n$.
    \end{theorem}
        
\subsection{Infinite-dimensional upper triangular matrices} We shall define the main object of our study. For, we follow \cite{Mes}. Let $\mathcal{V}$ be a $\mathbb{F}$-vector space admitting as a basis the set $\{v_i\mid i\in\mathbb{N}\}$. For each $i\in\mathbb{N}$, let $\mathcal{V}_i=\mathrm{Span}\{v_j\mid j\le i\}$. Then, we define the triangularizable algebra
$$
\mathrm{UT}_\infty=\{T\in\mathrm{End}\,\mathcal{V}\mid T\mathcal{V}_i\subseteq\mathcal{V}_i,\,\forall i\in\mathbb{N}\}.
$$
We can define the matrix units using the following. For each $i\le j$, let $e_{ij}$ be the linear operator $\mathcal{V}\to\mathcal{V}$ such that, in the basis, it satisfies $e_{ij}(v_\ell)=\delta_{j\ell}v_i$. Then, $e_{ij}e_{k\ell}=\delta_{jk}e_{i\ell}$. We define $\mathrm{UT}=\mathrm{Span}\{e_{ij}\mid 1\le i\le j\}$. The space $\mathrm{UT}$ is an associative algebra, and it is isomorphic to the direct limit of the finite dimensional upper triangular matrix algebras. It is relevant to mention that, for each $n\in\mathbb{N}$, the subspace $\mathrm{Span}\{e_{ij}\mid 1\le i\le j\le n\}$ is a subalgebra isomorphic to $\mathrm{UT}_n$. Thus, we shall identify $\mathrm{UT}_n$ as a subalgebra of $\mathrm{UT}$, for each $n\in\mathbb{N}$.

We can define a grading on $\mathrm{UT}$ by the following way. Given a map $\bar{g}:\mathbb{N}\to G$, we impose that each $e_{ij}$ is homogeneous of degree $\bar{g}(i)\bar{g}(j)^{-1}$. It is worth mentioning that the map $\bar{g}_n$, obtained by the restriction of $\bar{g}$ to the set $\{1,2,\ldots,n\}$, gives a $G$-grading on $\mathrm{UT}_n$, which agrees with the grading as a subspace of $\mathrm{UT}$. The classification of the group gradings on $\mathrm{UT}$ (and on $\mathrm{UT}_\infty$) is obtained in \cite{SY}. The authors prove that every group grading on $\mathrm{UT}$ (in the paper, they denote by $\mathrm{UT}_{\to\beta}$) is isomorphic to an elementary grading. Every group grading on $\mathrm{UT}_\infty$ (denoted by $\mathrm{UT}_\beta$ in the paper) has a finite support, and it is obtained by the (topological) closure of a group grading on $\mathrm{UT}$. More precisely, if $\mathrm{UT}=\bigoplus_{g\in G}\mathcal{A}_g$ is the decomposition of a finite (that is, with a finite support) $G$-grading on $\mathrm{UT}$, then $\mathrm{UT}_\infty=\bigoplus_{g\in G}\overline{\mathcal{A}_g}$ is a $G$-grading on $\mathrm{UT}_\infty$. Every $G$-grading on $\mathrm{UT}_\infty$ comes from this construction.

It is worth to discuss an alternative way to describe the elementary gradings on $\mathrm{UT}$. Given an elementary grading on $\mathrm{UT}$, we set $\varepsilon:\mathbb{N}\to G$ via $\varepsilon(i)=\deg_{G} e_{i,i+1}$. Conversely, given any map $\varepsilon:\mathbb{N}\to G$, we obtain a well-defined elementary grading on $\mathrm{UT}$ by imposing that each matrix unit $e_{i,i+1}$ is homogeneous and of degree $\varepsilon(i)$. Once again, the restriction of $\varepsilon$ to the set $I_{n-1}=\{1,2,\ldots,n-1\}$, denoted by $\varepsilon_n=\varepsilon|_{I_{n-1}}$, defines an elementary grading on $\mathrm{UT}_n$. It agrees with the grading induced on $\mathrm{UT}_n$ as a graded subspace of $\mathrm{UT}$.

\section{Graded polynomial identities}
\label{sec:grpolid}
\subsection{General results} Let $G$ be a group (with multiplicative notation and neutral element $1$), and let $\varepsilon:\mathbb{N}\to G$ define a group grading on $\mathrm{UT}$ via $\deg_{G} e_{i,i+1}=\varepsilon(i)$. Since each $\mathrm{UT}_n$ is a graded subalgebra of $\mathrm{UT}$ and $\mathrm{UT}=\bigcup_{n\in\mathbb{N}}\mathrm{UT}_n$, it is worth studying some general results that holds in this setting. The grading on $\mathrm{UT}_n$ is defined by the restriction of the map $\varepsilon$ to the set $I_{n-1}$, and we denote $\varepsilon_n=\varepsilon|_{I_{n-1}}$. We have the following results.
    \begin{lemma}\label{lem1}
        Let I be a directed system, and $\{\mathcal{A}_i\}_{i\in I}$ be a family of $G$-graded algebras such that $\mathcal{A}_i\subseteq\mathcal{A}_j$ (embedding of graded algebras) if $i\le j$. Let $\mathcal{A} = \lim\limits_{\longrightarrow}\mathcal{A}_{i}$. Then, $\mathcal{A}$ is a $G$-graded algebra, and
        $$
        \mathrm{Id}_G(\mathcal{A})=\bigcap_{i\in I}\mathrm{Id}_G(\mathcal{A}_i).
        $$
    \end{lemma}
    \begin{proof}
     Let $\mathcal{A}$ be the direct limit of $\{\mathcal{A}_i\}_{i\in I}$. Note that, for each $g\in G$, $\{(\mathcal{A}_i)_g\}$ is a direct system. Thus,
     $$ \mathcal{A}=\lim\limits_{\longrightarrow}\bigoplus_{g\in G}(\mathcal{A}_i)_g=\bigoplus_{g\in G}\lim\limits_{\longrightarrow}(\mathcal{A}_i)_g.
     $$
     So $\mathcal{A}$ has a $G$-grading as a vector space, where $\mathcal{A}_g=\lim\limits_{\longrightarrow}(\mathcal{A}_i)_g$. Now, let $a$, $b\in\mathcal{A}$ be homogeneous. Then, we can find $\mathcal{A}_i$ such that $a$, $b\in\mathcal{A}_i$. It implies that $ab$ is homogeneous of the correct degree, so that $\mathcal{A}$ is a $G$-graded algebra.
     
        Since $\mathcal{A}_i\subseteq\mathcal{A}$, for all $i\in I$, one has $\mathrm{Id}_G(\mathcal{A})\subseteq\bigcap_{i\in I}\mathrm{Id}_G(\mathcal{A}_i)$. Conversely, let $f(x_{1}^{(g_1)},\ldots,x_{m}^{(g_m)}) \in \bigcap_{i \in I}\mathrm{Id}_{G}(\mathcal{A}_{i}, \Gamma_{i})$. Suppose that $f \notin A$. Then, there exist $a_{1}$, \dots, $a_{m}\in\mathcal{A}$, homogeneous and $\deg_G a_i=g_i$, such that $f(a_{1}, \dots, a_{m}) \neq 0$. Moreover, we can find $i\in I$ such that $a_{1}$, \dots, $a_{m} \in\mathcal{A}_{i}$. Therefore, $f\notin\mathrm{Id}_G(\mathcal{A}_i)$, a contradiction. Hence, the equality holds valid.
    \end{proof}

Consider the same situation as in \Cref{lem1}, and let $\mathbb{F}_\mathcal{A}(X)$ and $\mathbb{F}_i(X)$ be the relatively free $G$-graded algebra in the variety $\mathrm{Var}(\mathcal{A})$ and $\mathrm{Var}(\mathcal{A}_i)$, respectively, for each $i\in I$. For each $i\le j$, since $\mathrm{Id}_G(\mathcal{A}_j)\subseteq\mathrm{Id}_G(\mathcal{A}_i)$, we have a surjective algebra homomorphism $\varphi_{ji}:\mathbb{F}_j(X)\to\mathbb{F}_i(X)$, extending the identity map $X\to X$.
\begin{Prop}\label{lem2}
Using the notation above, we have
$$
\mathbb{F}_\mathcal{A}(X)\cong\lim\limits_{\longleftarrow}\mathbb{F}_i(X).
$$
\end{Prop}
\begin{proof}
We shall use the explicit construction of the inverse limit algebra, and prove that it is the relatively free $G$-graded algebra in the variety $\mathrm{Var}(\mathcal{A})$, freely generated by $X$. For, let
$$
\mathcal{B}=\left\{(a_i)_{i\in I}\in\prod_{i\in I}\mathbb{F}_i(X)\mid \varphi_{ji}(a_j)=a_i,\,\forall i\le j\right\}.
$$
Then, we have a bijection
$$
x\in X\mapsto (x)_{i\in I}\in\mathcal{B}.
$$
Denote such image by $\bar{X}$. Since $\mathrm{Id}_G(\mathcal{A})\subseteq\mathrm{Id}_G(\mathcal{A}_i)$, for each $i\in I$, there exists a surjective algebra homomorphism $\varphi_i:\mathbb{F}_\mathcal{A}(X)\to\mathbb{F}_i(X)$, extending the identity map $X\to X$. Moreover, one has $\varphi_i=\varphi_{ji}\circ\varphi_j$, for all $i\le j$. Hence, the map
$$
\varphi:f\in\mathbb{F}_\mathcal{A}(X)\mapsto(\varphi_i(f))_{i\in I}\in\mathcal{B}
$$
is well-defined and surjective. Let $f\in\mathrm{Ker}\,\varphi$. It means that $\varphi_i(f)=0$, for all $i\in I$. That is, $f\in\bigcap_{i\in I}\mathrm{Id}_G(\mathcal{A}_i)=\mathrm{Id}_G(\mathcal{A})$ (from \Cref{lem1}). It implies that $f=0$, so $\varphi$ is an algebra isomorphism.
\end{proof}

Now, we shall explore the connection between the graded polynomial identities of $\mathrm{UT}$ and of $\mathrm{UT}_\infty$. So, assume that $\Gamma$ is a $G$-grading on $\mathrm{UT}$ having a finite support. Then, $\Gamma$ extends to a $G$-grading $\bar{\Gamma}$ on $\mathrm{UT}_\infty$. In this case, we have:
\begin{Prop}
Let $\bar{\Gamma}$ be a $G$-grading on $\mathrm{UT}_\infty$ and $\Gamma$ its restriction to $\mathrm{UT}$. Then
$\mathrm{Id}_G(\mathrm{UT},\Gamma)=\mathrm{Id}_G(\mathrm{UT}_\infty,\bar{\Gamma})$.
\end{Prop}
\begin{proof}
Since $\mathrm{UT}\subseteq\mathrm{UT}_\infty$, then clearly $\mathrm{Id}_G(\mathrm{UT}_\infty,\bar{\Gamma})\subseteq\mathrm{Id}_G(\mathrm{UT},\Gamma)$. Conversely, let $f=f(x_1,\ldots,x_m)\in\mathrm{Id}_G(\mathrm{UT},\Gamma)$ and $a_1$, \dots, $a_m\in\mathrm{UT}_\infty$. We may assume that each matrix unit is homogeneous with respect to $\Gamma$. Let $v_n\in\mathcal{V}$ be a basis element and consider the projection $p_n:\mathrm{UT}_\infty\to\mathrm{UT}_n$. Since $p_n$ is a graded algebra homomorphism and $\mathrm{Ker}\,p_n\subseteq\mathrm{Ann}_{\mathrm{UT}_\infty}(v_n)$, we have
$$
f(a_1,\ldots,a_m)v_n=p_n(f(a_1,\ldots,a_m))v_n=f(p_n(a_1),\ldots,p_n(a_n))v_n=0.
$$
Thus, $f(a_1,\ldots,a_m)$ is the zero operator on $\mathcal{V}$. Hence, $f\in\mathrm{Id}_G(\mathrm{UT}_\infty,\bar{\Gamma})$.
\end{proof}
\begin{remark}
An alternative proof of the converse is the following. Consider $f=f(x_1,\ldots,x_m)\in\mathrm{Id}_G(\mathrm{UT},\Gamma)$ and $a_1$, \dots, $a_m\in\mathrm{UT}_\infty$. Then, we can find a net $(a_{1i},\ldots,a_{mi})$ in $\mathrm{UT}$ converging to $(a_1,\ldots,a_m)$. Since
$$
(x_1,\ldots,x_m)\in\mathrm{UT}_\infty\times\cdots\times\mathrm{UT}_\infty\mapsto f(x_1,\ldots,x_m)\in\mathrm{UT}_\infty
$$
is obtained from a linear combination of composition of the multiplication of the algebra, it is a continuous map. Thus,
$$
f(a_1,\ldots,a_m)=\lim \underbrace{f(a_{1i},\ldots,a_{mi})}_0=0.
$$
Hence, $f\in\mathrm{Id}_G(\mathrm{UT}_\infty,\bar{\Gamma})$, as required.
\end{remark}
The latter proposition tells us that it is enough to classify the graded polynomial identities of $\mathrm{UT}$, and from it, we automatically obtain the graded polynomial identities of $\mathrm{UT}_\infty$.
    
\subsection{Graded polynomial identities of $\mathrm{UT}$}
   Recall that $\varepsilon:\mathbb{N}\to G$ defines a $G$-grading on $\mathrm{UT}$ via $\deg_{G}(e_{i,i+1})=\varepsilon(i)$. We shall see $\varepsilon = (\varepsilon_{i})_{i\in\mathbb{N}}$ as a sequence of elements of $G$. First, we repeat the notion of good and bad sequence from \cite[Definition 2.1]{VinKoVa2004}:
    
    \begin{definition}
         We call a finite sequence $\eta = (\eta_{1},\eta_{2},\ldots,\eta_{m})$, where $m\ge0$, of elements from $G$ an \emph{$\varepsilon$-good sequence} if there exists a sequence of $m$ matrix units $(r_{1},r_{2},\ldots,r_{m})$ in the Jacobson radical of $\mathrm{UT}$ such that $r_{1}r_{2}\cdots r_{m} \neq 0$ and $\deg_{G}(r_{i}) = \eta_{i}$, for all $i = 1,\dots,m$. Otherwise, we call $\eta$ an \emph{$\varepsilon$-bad sequence}. We denote by $\mathcal{B}(\varepsilon)$ the set of all $\varepsilon$-bad sequences, and $\mathcal{G}(\varepsilon)$ the set of all $\varepsilon$-good sequences. Similarly, we will denote by $\mathcal{B}(\varepsilon_{n})$ the set of all $\varepsilon_{n}$-bad sequences and $\mathcal{G}(\varepsilon_{n})$ the set of all $\varepsilon_{n}$-good sequences.
    \end{definition}

    Now, we give a description of the set $\mathcal{G}(\varepsilon)$.
    \begin{Prop}\label{goodsequencesbadsuite}
        The set $\mathcal{G}(\varepsilon)$ is inductively constructed as following:
        \begin{enumerate}
            \item All finite consecutive subsequences of $\varepsilon$ are in $\mathcal{G}(\varepsilon)$, and
            \item If $(\eta_{1},\eta_{2},\ldots,\eta_{m}) \in \mathcal{G}(\varepsilon)$, then $(\eta_{1},\eta_{2},\ldots,\eta_{p-1},\eta_{p}\eta_{p+1},\eta_{p+2},\ldots,\eta_{m}) \in \mathcal{G}(\varepsilon)$ for all $1 \leq p < m$
        \end{enumerate}
    \end{Prop}
    \begin{proof}
        Let $(g_{j},g_{j+1},\ldots,g_{l})$ be a consecutive subsequence of $\varepsilon$. Since for all $i$ we have $\deg_{G}(e_{i,i+1}) = g_{i}$ and $e_{j,j+1}\cdots e_{l,l+1} \neq 0$, we get that $(g_{j},g_{j+1},\dots,g_{l})$ is $\varepsilon$-good.
        
        Now, let $(\eta_{1},\eta_{2},\dots,\eta_{m}) \in \mathcal{G}(\varepsilon)$. Then, there exists a sequence of $m$ matrix units $(r_{1},r_{2},\ldots,r_{m}) \in J(\mathrm{UT})$  such that $r_{1}r_{2}\cdots r_{m} \neq 0$ and $\deg_{G}(r_{i}) = \eta_{i}$, for all $i$. Now, $\deg_{G}(r_{p}r_{p+1}) = \eta_{i}\eta_{i+1}$ for all $1 \leq p < m$. Therefore, by definition,
        $(\eta_{1},\eta_{2},\ldots,\eta_{p-1},\eta_{p}\eta_{p+1},\eta_{p+2},\ldots,\eta_{m}) \in \mathcal{G}(\varepsilon)$.
        
        Finally, let $\eta = (\eta_{1},\eta_{2},\ldots,\eta_{m}) \in \mathcal{G}(\varepsilon)$. Once again, there exist matrix units $e_{i_{1},i_{2}}$, \dots, $e_{i_{m},i_{m+1}}$ such that $\deg_{G}(e_{i_{p},i_{p+1}}) = \eta_{p}$ and $e_{i_{1},i_{2}}\cdots e_{i_{m},i_{m+1}} \neq 0$. However, $e_{i_{p},i_{p+1}} = e_{i_{p},i_{p}+1}\cdots e_{i_{p+1}-1,i_{p+1}}$, for each $p$. Therefore, by induction, there exists a finite consecutive subsequence of $\varepsilon$ that generates $\eta$ using the process (2) several times.
    \end{proof}

    As a consequence we have the following statement:
    \begin{Prop}
        \label{monoidGenerator}
        Let $G$ be a group and let $\varepsilon:\mathbb{N}\to G$ define a $G$-grading on $\mathrm{UT}$. If $A$ is a set of generators of G as a monoid then every finite sequence of elements of $A$ is $\varepsilon$-good if and only if every finite sequence of elements of $G$ is $\varepsilon$-good.\qed
    \end{Prop}

    Now, we construct the following polynomials:
    \begin{definition}[{\cite{VinKoVa2004} and \cite[Definition 3.19]{GR2020}}]
        \label{badpolynomial}
         We denote by $I(\varepsilon)$ the ideal of graded polynomial identities generated by all polynomials $f_\eta=f_{1}\cdots f_{m}$ where $\eta=(\deg_{G}(f_{1}),\ldots,\deg_{G}(f_{m}))$ is an $\varepsilon$-bad sequence and each $f_{i}$ has one of the following forms:
        \begin{enumerate}
            \item $f_{i} = z_{i}$,
            \item $f_{i} = [y_{2i},y_{2i+1}]$,
            \item $f_{i} = y_{2i}^{q} - y_{2i}$, if $\mathbb{F}$ is a finite field with $q$ elements.
        \end{enumerate}
    \end{definition}
Following the finite-dimensional case (see \cite[Proposition 2.2]{VinKoVa2004} and \cite[Lemma 3.20]{GR2020}), it follows that a sequence is $\varepsilon$-bad if and only if the polynomials constructed above are $G$-graded polynomial identities of $(\mathrm{UT},\varepsilon)$. It means that $I(\varepsilon) \subseteq\mathrm{Id}_{G}(\mathrm{UT}, \varepsilon)$.

    First, we can describe the $\varepsilon$-bad and $\varepsilon$-good sequences in terms of the $\varepsilon_n$-bad and $\varepsilon_n$-good sequences.
    \begin{lemma}
        \label{propositionBadGood}
        The following relations are valid:
        \begin{enumerate}
        \item $\mathcal{B}(\varepsilon) = \bigcap\limits_{n \in \mathbb{N}}\mathcal{B}(\varepsilon_{n})$,
        \item $\mathcal{G}(\varepsilon)= \bigcup\limits_{n\in\mathbb{N}}\mathcal{G}(\varepsilon_n)$.
        \end{enumerate}
    \end{lemma}
    \begin{proof}
        Let $(\eta_{1},\eta_{2},\ldots,\eta_{m}) \in \mathcal{B}(\varepsilon)$, and suppose that $(\eta_{1},\eta_{2},\ldots,\eta_{m}) \notin \mathcal{B}(\varepsilon_{n})$ for some $n \in \mathbb{N}$. Then, there exists $r_{1},r_{2},\ldots,r_{m} \in J(UT_{n})$ such that $r_{1}r_{2}\cdots r_{m} \neq 0$. Since $J(\mathrm{UT}_n)\subseteq J(\mathrm{UT})$, we get a contradiction. Therefore, $(\eta_{1},\eta_{2},\dots,\eta_{m}) \in \bigcap\limits_{n \in\mathbb{N}}\mathcal{B}(\varepsilon_{n})$.
        
        Conversely, let $\eta=(\eta_{1},\eta_{2},\ldots,\eta_{m}) \in \bigcap\limits_{n \in \mathbb{N}}\mathcal{B}(\varepsilon_{n})$. Assume that there exist $r_1$, \dots, $r_m\in J(\mathrm{UT})$ such that $r_1\cdots r_m\ne0$ and $\deg_{G}(r_i)=\eta_i$, for $i=1,2,\ldots,m$. Then, we can find $n\in\mathbb{N}$ such that $r_1$, \dots, $r_m\in J(\mathrm{UT_n})$, so $\eta$ is $\varepsilon_n$-good, a contradiction.

        Now, consider all the finite subsequences of elements of $G$. Since every sequence is either $\varepsilon$-good or $\varepsilon$-bad, then we can take complements to obtain
        $$    \mathcal{G}(\varepsilon)=\mathcal{B}(\varepsilon)^c=\left(\bigcap\limits_{n\in\mathbb{N}}\mathcal{B}(\varepsilon_n)\right)^c=\bigcup_{n\in\mathbb{N}}\mathcal{B}(\varepsilon_n)^c=\bigcup_{n\in\mathbb{N}}\mathcal{G}(\varepsilon_n).
        $$
        This proves the second assertion.
    \end{proof}

    Digressing for a moment, we will state some properties in the (non-graded) free associative algebra. First, recall the concept of standard commutator given by Hall in \cite{H1950}. For it, let $X = \{x_{1},x_{2},\ldots\}$ be a set of variables with an ordering where $x_{i_{1}} < x_{i_{j}}$, if $i_{1} < i_{j}$. We will extend this ordering to the set of all Lie words on $X$ deg-lexicographically. It means that, given $u=[u_{1},u_2]$ and $v=[v_{1},v_2]$, we have $u < v$ if $\deg(u) < \deg(v)$, or $\deg (u) = \deg(v)$ and either $u_1<v_1$, or $u_1=v_1$ and $u_2<v_2$. Then we define inductively:
    \begin{enumerate}
        \item all elements of $X$ are standard commutators of degree $1$.
        \item $[u, v]$ is a standard commutator of degree $n = \deg(u) + \deg(v)$ if:
        \begin{enumerate}
            \item $u$ and $v$ are standard commutators with $u>v$, and
            \item if $u = [w, z]$, then $z \leq v$.
        \end{enumerate}
    \end{enumerate}

    Hall showed that the set of all standard commutators is a basis of the free Lie algebra generated by $X$. Thus, combining with PBW we obtain a basis of the free associative algebra.

    The following will be useful for our purposes.
    
    \begin{definition}
        Let $c$ be a standard commutator. We define the \emph{weight of $c$}, denoted by $w(c)$, inductively by the following:
        \begin{enumerate}
            \item if $\deg(c) = 1$ (that is, $c$ is a variable), then $w(c) = 0$,
            \item if $c=[x,y]$, then $w(c) = 1$,
            \item if $c=[c_{1},c_{2}]$, where $\deg(c) > 2$, and $c_{1}$ and $c_{2}$ are standard commutators, then $w(c) = w(c_{1}) + w(c_{2})$.
        \end{enumerate}
    \end{definition}

    It is worth noting that the standard commutators of weight $1$ are of the kind $[x_{i_1},x_{i_2},\ldots,x_{i_m}]$, with $i_1>i_2\le i_3\cdots\le i_m$.
    \begin{lemma}
        \label{lem3.15}
        If $c = [c_{1},c_{2}]$ is a standard commutator with $w(c) > 1$, then we have $0 < w(c_{i}) < w(c)$, for $i=1,2$.
    \end{lemma}
    \begin{proof}
        Since $c = [c_{1},c_{2}]$, we have $w(c_{1}) + w(c_{2}) = w(c)$. Then, it is enough to prove that $w(c_{1}), w(c_{2}) > 0$. We shall do this by induction on $\deg(c)$. For $\deg(c) \leq 3$ the statement is vacuous. Suppose that $\deg(c)>3$ and the lemma holds valid for all standard commutators $c'$ with $\deg(c')<\deg(c)$. Then, if $c$ is a standard commutator with $w(c) > 1$, we can write $c = [c_{1},c_{2}]$ where $c_{1}$ and $c_{2}$ are standard commutators satisfying $w(c_{1}) + w(c_{2}) = w(c)$, $\deg(c_{1}) + \deg(c_{2}) = \deg(c)$, and $\deg(c_{1}) \geq \deg(c_{2})$. Then, $\deg(c_{1}) > 1$ and therefore, $w(c_{1}) > 0$. Suppose that $w(c_{2}) = 0$. Thus, $w(c_{1}) = w(c)$, $\deg(c_{2}) = 1$ and $\deg(c_{1}) = \deg(c) - 1$. Then, by induction hypothesis, $c_{1} = [d_{1},d_{2}]$ where $d_{1}, d_{2}$ are standard commutators with $w(d_{1}), w(d_{2}) > 0$ and $\deg(c_{2}) \geq \deg(d_{2}) > 2$. This contradiction proves that $w(c_{2}) > 0$.
    \end{proof}
    
    \begin{lemma}
        \label{semistandardGenerator}
        Every standard commutator $c$ such that $\deg(c) \geq 2$ can be written as a linear combination of products of $w(c)$ standard commutators with weight $1$.
    \end{lemma}
    \begin{proof}
        We will prove the lemma by induction on $w(c)$. If $w(c) = 1$ we are done. Let $w(c) > 1$, and suppose that the lemma holds for standard commutators $c'$ satisfying $w(c') <w(c)$. Write $c = [c_{1},c_{2}]$ where $c_{1}$ and $c_{2}$ are standard commutators with $0 < w(c_{1}), w(c_{2}) < n$ (\Cref{lem3.15}). Then, by induction hypothesis, $c_{1}$ and $c_{2}$ are linear combination of products of $w(c_{1})$ and $w(c_{2})$ standard commutators of weight $1$, respectively. Hence, $c = [c_{1},c_{2}] = c_{1}c_{2} - c_{2}c_{1}$ is a linear combination of products of $w(c_1)+w(c_2)=w(c)$ standard commutators of weight $1$.
    \end{proof}

\begin{definition}[{\cite{VinKoVa2004}}]
    \label{def:semistandard}
    A polynomial $c \in \mathbb{F}\langle X^G\rangle$ is called \emph{semistandard commutator} if it has one of the following forms:
    \begin{enumerate}
        \item $c = [z_{i_{1}},y_{i_{2}},\ldots,y_{i_{s}}]$,  $i_{2} \leq \ldots \leq i_{s}$, or
        \item $c = [y_{i_{1}},y_{i_{2}},\ldots,y_{i_{s}}]$, $i_{1} > i_{2} \leq \ldots \leq i_{s}$ and $s \geq 2$.
    \end{enumerate}
\end{definition}

We extend the notion of weight to semistandard commutators by assigning the weight $1$ to them. Next, we extend the notion of a standard commutator to graded polynomials.
\begin{definition}
    A polynomial $c \in \mathbb{F}\langle X^G\rangle$ is called a \emph{$G$-standard commutator} if it is not a variable of trivial degree, and either all variables has homogeneous trivial degree and it is a standard commutator (in the sense of Hall), or it is a semistandard commutator (\Cref{def:semistandard}).
\end{definition}

    Then, combining the previously defined weight for standard and semistandard commutators, we have a well-defined weight for all $G$-standard commutators.
    
    For the following, write $(g^{(n)}) = (\underbrace{g, \ldots, g}_{n \text{ times}})$, where $g \in G$ and $n \in \mathbb{N}$. We define $(g^{(0)})$ to be the empty sequence.

    \begin{lemma}
        \label{GstandardGenerator}
        The vector space $\mathbb{F}\langle X^{G}\rangle/{I(\varepsilon)}$ is spanned by the set of all polynomials
        $$f = y_{i_{1}}\ldots y_{i_{s}}c_{1}\ldots c_{m}$$
        where $i_{1} \leq \ldots \leq i_{s}$ and $c_{1}, \ldots, c_{m}$ are $G$-standard commutators such that the sequence $(\deg_{G}(c_{1})^{(w(c_{1}))},\ldots,\deg_{G}(c_{m})^{(w(c_m))})$ is $\varepsilon$-good.
    \end{lemma}
    \begin{proof}
    We have to proof that these polynomials generate $\mathbb{F}\langle X\rangle$ modulo $I(\varepsilon)$. Let $p = x_{j_{1}}^{(g_{j_{1}})} \ldots x_{j_{n}}^{(g_{j_{n}})} \in \mathbb{F}\langle X^{G}\rangle$. The proof will be by induction on $n$.

    For $n = 1$, $p$ is just a variable and we are done. Suppose that the lemma holds for all products of at most $n-1$ variables. By induction hypothesis, $p = qx_{j_{n}}^{(g_{j_{n}})}$ where $q = \sum\limits_{i}y_{i,1}\ldots y_{i,s}c_{1}\ldots c_{m}$ is a linear combination of polynomials of the required form. If $g_{j_{n}} \neq 1$ there is nothing to do. So, suppose that $g_{j_{n}} = 1$. Then, $p = \sum\limits_{i}y_{i,1}\cdots y_{i,s}c_{1}\ldots c_{m}y_{j_{n}}$. For each summand $p_{i} = y_{i,1}\cdots y_{i,s}c_{1}\ldots c_{m}y_{j_{n}}$ of $p$ we have 
    \[
    p_{i} = y_{i,1}\ldots y_{i,s}c_{1}\cdots c_{m-1}y_{j_{n}}c_{m} + y_{i,1}\cdots y_{i,s}c_{1}\ldots c_{m-1}[c_{m},y_{j_{n}}].
    \]
    By induction, the first summand of each $p_{i}$ is a linear combination of the mentioned polynomials. In the second term we have a polynomial of the wanted form multiplied by a commutator $c = [c_{m},y_{j_{n}}]$. If $\deg_{G}c = 1$ then, from the Hall basis, $c$ is as a linear combination of the polynomials of the wanted form. So, assume that $\deg_{G}c = g \neq 1$. Then, $c_m = [x_{j_{1}}^{(g)},y_{j_{2}},\ldots,y_{j_{n-1}}]$. If $j_{n-1} \leq j_{n}$, then there is nothing to do. Otherwise, we have that
    $$
        c = [x_{j_{1}}^{(g)},y_{j_{2}},\ldots,y_{j_{n}},y_{j_{n-1}}] + [[x_{j_{1}}^{(g)},y_{j_{2}},\ldots,y_{j_{n-2}}],[y_{j_{n-1}}, y_{j_{n}}]]
    $$
    and the result follows by induction.
\end{proof}

We shall improve the statement of the previous lemma. For, we need:
    \begin{lemma}
        \label{infiniteFieldBasis}

        The vector space $\mathbb{F}\langle X^{G}\rangle/{I(\varepsilon)}$ is spanned by the set of all polynomials
        \begin{equation}\label{seqsemistand}
        f = y_{i_{1}}\cdots y_{i_{s}}c_{1}\cdots c_{m}
        \end{equation}
        where $i_{1} \leq \ldots \leq i_{s}$ and $c_{1}, \ldots, c_{m}$ are semistandard commutators such that the sequence $(\deg_{G}(c_{1}),\ldots,\deg_{G}(c_{m}))$ is $\varepsilon$-good.
    \end{lemma}
    \begin{proof}
        We need to prove that every polynomial of \Cref{GstandardGenerator} is a linear combination of the polynomials \eqref{seqsemistand}. However, a $G$-standard commutator is either a semistandard commutator, or all variables has homogeneous trivial degree and it is a standard commutator. Thus, by \Cref{semistandardGenerator} a standard commutator of weight $w$ is a linear combination of products of $w$ standard commutators of weight $1$. In particular, a standard commutator of weight 1 is a semistandard commutator.
    \end{proof}

    Next, to include the case where the base field is finite, we need a variation of the definition of E-normal polynomial, given in \cite[Definition 3.7]{GR2020}.
    \begin{definition}
        If $\mathbb{F}$ is a finite field with $q$ elements, then a polynomial $c \in \mathbb{F}\langle X^G\rangle$ is called \emph{E-normal polynomial} if it is either a semistandard commutator or of the form 
        $$c = [y_{i_{1}}^{q} - y_{i_{1}},y_{i_{2}},\ldots,y_{i_{s}}],\quad i_{2} \leq \ldots \leq i_{s},\, i_{1} \notin \{i_{2},\ldots,i_{s}\}$$
    \end{definition}

    \begin{lemma}
        \label{ENormalBasis}
        Let $\mathbb{F}$ be a finite field with $q$ elements. The vector space $\mathbb{F}\langle X^{G}\rangle/{I(\varepsilon)}$ is spanned by the set of all polynomials
        $$f = y_{1}^{a_{1}}\cdots y_{s}^{a_{s}} c_{1}\cdots c_{m}$$
        where $0 \leq  a_{1},\ldots,a_{s} < q$, $(\deg_{G}(c_{1}),\ldots,\deg_{G}(c_{m}))$ is $\varepsilon$-good, and each $c_{i}$ is an E-normal polynomial.
    \end{lemma}
    \begin{proof} 
        By \Cref{infiniteFieldBasis}, we already know that $\mathbb{F}\langle X^{G}\rangle/{I(\varepsilon)}$ is spanned by the polynomials of the form $g = y_{1}^{a_{1}}\cdots y_{s}^{a_{s}} c_{1}\cdots c_{m}$, where $a_{1},\ldots,a_{s} \geq 0$, $(\deg_{G} c_{1},\ldots,\deg_{G} c_{m})$ is $\varepsilon$-good, and each $c_{i}$ is a semistandard commutator. By \cite[Claim 3 of the proof of Lemma 3.5]{GR2020}, the exponents are bounded.
    \end{proof}

    Let $f \in \mathbb{F}\langle X^{G}\rangle$, $y$ be a variable with trivial $G$-degree and $d \in \mathbb{N}$. We denote
    \[
    [f, y^{(d)}] = [f, \underbrace{y, \ldots, y}_{d \text{ times}}]. 
    \]
    We extend the notation and write $[f, y^{(0)}] = f$.
    
    The following result can be found in \cite[Lemma 2.5]{S1981} and \cite[Lemma 3.13]{GR2020}.

    \begin{lemma}
        \label{SiderovLemma}
        If $\mathbb{F}$ is a finite field with $q$ elements, $f \in \mathbb{F}\langle X^{G}\rangle$ and $y$ is a variable of trivial homogeneous degree, then
        $$[f,y^{(q)}] = [f,y^{q}].$$
    \end{lemma}

    \begin{definition}[{\cite[Definition 3.15]{GR2020}}]\label{Lnormal}
        If $\mathbb{F}$ is a finite field with $q$ elements, then an E-normal polynomial $c \in \mathbb{F}\langle X^G\rangle$ is called \emph{L-normal polynomial} if it has one of the following forms:
        \begin{enumerate} 
            \item $c = [z_{1},y_{2}^{(d_{2})},\ldots,y_{s}^{(d_{s})}]$ with $0 \leq d_{2},\ldots,d_{s} < q$, or
            \item $c = [y_{i},y_{1}^{(d_{1})},\ldots,y_{i}^{(d_{i}-1)},\ldots,y_{s}^{(d_{s})}]$ with $0 \leq d_{1},\ldots,d_{s} < q$, or
            \item $c = [{y_{i}}^{q} - y_{i},y_{1}^{(d_{1})},\ldots,y_{i-1}^{(d_{i-1})},y_{i+1}^{(d_{i+1})},\ldots,y_{s}^{(d_{s})}]$ with $0 \leq d_{1},\ldots,d_{s} < q$.
        \end{enumerate}
    \end{definition}
    
    \begin{lemma}
        \label{spanSetFiniteField}
        Let $\mathbb{F}$ be a finite field with $q$ elements. The vector space $\mathbb{F}\langle X^{G}\rangle/{I(\varepsilon)}$ is spanned by the set of all polynomials
        $$f = y_{1}^{a_{1}}\cdots y_{s}^{a_{s}} c_{1}\cdots c_{m}$$
        where $0 \leq  a_{1},\ldots,a_{s} < q$, $(\deg_{G}(c_{1}),\ldots,\deg_{G}(c_{m}))$ is $\varepsilon$-good, and each $c_{i}$ is an L-normal polynomial.
    \end{lemma}
\begin{proof}
    By \Cref{ENormalBasis}, we need to prove that an E-normal polynomial $c$ is linear combination of products of L-normal polynomials, modulo $I(\varepsilon)$. If $c$ has degree less or equal than $q$, then either $c = y^{q} - y$ for some variable $y$ of $G$-degree $1$, or every variable on $c$ appears less than $q$ times. In both cases, $c$ is L-normal.
        
    Now, let $n \geq q$ and suppose that the lemma holds for all $c$ with $\deg(c) \leq n$. Let $c$ be an E-normal commutator with $\deg(c) = n + 1$. If $c$ is an L-normal polynomial we are done. So, suppose that there is some variable $y$ that appears in $c$ at least $q$ times in the last position. Thus, we have $c = [f, y^{(q)}]$ and by \Cref{SiderovLemma}, 
        $$c = [f, y^{q}] = [f, y^{q} - y] + [f, y] = f(y^{q} - y) - (y^{q} - y)f + [f, y].$$
        
        So by induction, $c$ is a linear combination of L-normal polynomials, modulo $I(\varepsilon)$.
        
        Now, suppose that the last variable in $c$ appears less than $q$ times and some variable $y$ appears in $c$ at least $q$ times. Let $l \leq n - q$ be an index such that $y = y_{l} = \ldots = y_{l+q-1} \neq y_{l+q}$, then we have
        \begin{equation*}
            \begin{split}
                c &= [f, y^{(q)}, y_{i_{l+q}},\ldots,y_{i_{n}}] = [f, y^{q} - y, y_{i_{l+q}},\ldots,y_{i_{n}}] + [f, y, y_{i_{l+q}},\ldots,y_{i_{n}}]\\
                &= [f(y^{q} - y) - (y^{q} - y)f, y_{i_{l+q}},\ldots,y_{i_{n}}] + [f, y, y_{i_{l+q}},\ldots,y_{i_{n}}]\\ 
                &= [f(y^{q} - y), y_{i_{l+q}},\ldots,y_{i_{n}}] - [(y^{q} - y)f, y_{i_{l+q}},\ldots,y_{i_{n}}] + [f, y, y_{i_{l+q}},\ldots,y_{i_{n}}]
            \end{split}
        \end{equation*}
        Recall the formula
        $$[ab,x_{1},\ldots,x_{m}] = \sum_{\{j_{1},\ldots,j_{r}\} \subseteq \{1,\ldots,m\}} [a, x_{j_{1}},\ldots,x_{j_{r}}][b, x_{k_{1}},\ldots,x_{k_{s}}]$$
        where $j_{1} < \ldots < j_{r}$, $k_{1} < \ldots < k_{s}$ and $\{j_{1},\ldots,j_{r}\}\cup\{k_{1}, \ldots, k_{s}\} = \{1,\ldots,m\}$. If $f$ is an E-normal polynomial, we conclude that $c$ is a a linear combination of products of E-normal polynomials with degree less than $n + 1$, modulo $I(\varepsilon)$. Now, assume that $f$ is a variable with trivial $G$-degree. If $f \leq y_{l+q}$, then
        $$c = [f, y, y_{i_{l+q}},\ldots,y_{i_{n}}] - [y^{q} - y, f, y_{i_{l+q}},\ldots,y_{i_{n}}]$$
        and we are done. On the other hand, if $f > y_{l+q}$ then let $t$ be the biggest index such that $f > y_{t}$. So,
        \begin{equation*}
            \begin{split}
                d &= [f(y^{q} - y), y_{i_{l+q}},\ldots,y_{i_{t}},\ldots,y_{i_{n}}] - [(y^{q} - y)f, y_{i_{l+q}},\ldots,y_{i_{t}},\ldots,y_{i_{n}}]\\
                &= [\sum_{\{j_{1},\ldots,j_{r}\} \subseteq \{i_{l+q},\ldots,i_{t}\}} [f, y_{j_{1}},\ldots,y_{j_{r}}][y^{q}-y, y_{k_{1}},\ldots,y_{k_{s}}],y_{i_{t+1}},\ldots,y_{i_{n}}]\\
                &- [\sum_{\{j_{1},\ldots,j_{r}\} \subseteq \{i_{l+q},\ldots,i_{t}\}} [y^{q}-y, y_{k_{1}},\ldots,y_{k_{s}}][f,y_{j_{1}},\ldots,y_{j_{r}}],y_{i_{t+1}},\ldots,y_{i_{n}}]\\
                &= [d' + f[y^{q} - y, y_{i_{l+q}},\ldots,y_{i_{t}}] - [y^{q} - y, y_{i_{l+q}},\ldots,y_{i_{t}}]f,y_{i_{t+1}},\ldots,y_{i_{n}}]\\
                &= [d',y_{i_{t+1}},\ldots,y_{i_{n}}] - [y^{q} - y, y_{i_{l+q}},\ldots,y_{i_{t}},f,y_{i_{t+1}},\ldots,y_{i_{n}}]
            \end{split}
        \end{equation*}.
        
        Since $d'$ is a linear combination of products of two E-normal polynomials and $y_{i_{l+q}} \leq \ldots \leq y_{i_{t}} < f \leq y_{i_{t+1}} \leq \ldots \leq y_{i_{n}}$ we have by induction that the proposition holds for the first summand of $d$ and repeating the procedure we conclude that it holds also for the second summand. By consequence, it holds for $c$ as well.
        
        Finally, suppose that $c$ has a variable $y$ of trivial $G$-degree in the first position such that it appears in $c$ exactly $q$ times and all the other variables appears less than $q$ times. Then, we have that $y$ appears $q - 1$ times at the position $l + 1$ in $c$, where $2 < l + 1 \leq n - q + 3$. So, $c$ has the form $c = [y, y_{i_{1}},\ldots,y_{i_{l-1}},y^{(q-1)},y_{i_{l+q-1}},\ldots,y_{i_{n}}]$. If $l = 2$, we have
        \begin{equation*}
            \begin{split}
                c &= [y, y_{i_{1}},y^{(q-1)},y_{i_{q+1}},\ldots,y_{i_{n}}] = - [y_{i_{1}},y^{(q)},y_{i_{q+1}},\ldots,y_{i_{n}}]\\
                &= - [y_{i_{1}},y^{q},y_{i_{q+1}},\ldots,y_{i_{n}}] = [y^{q},y_{i_{1}},y_{i_{q+1}},\ldots,y_{i_{n}}]\\
                &= [y^{q}-y,y_{i_{1}},y_{i_{q+1}},\ldots,y_{i_{n}}] + [y,y_{i_{1}},y_{i_{q+1}},\ldots,y_{i_{n}}]
            \end{split}
        \end{equation*}
        and we are done. If $l>2$, then
        \begin{equation*}
            \begin{split}
                c &= [f,y_{i_{l-1}},y^{(q-1)},y_{i_{l+q-1}},\ldots,y_{i_{n}}] = [f y_{i_{l-1}} - y_{i_{l-1}}f, y^{(q-1)},y_{i_{l+q-1}},\ldots,y_{i_{n}}]\\
                &= [\sum_{r \leq q-1} [f, y^{(r)}][y_{i_{l-1}}, y^{(q-r-1)}] - [y_{i_{l-1}}, y^{(q-r-1)}][f, y^{(r)}],y_{i_{l+q-1}},\ldots,y_{i_{n}}]\\
                &= [[f,y^{(q-1)}]y_{i_{l-1}} - y_{i_{l-1}}[f,y^{(q-1)}],y_{i_{t+1}},\ldots,y_{i_{n}}]\\ 
                &+ [\sum_{r < q-1} [f, y^{(r)}][y,y_{i_{l-1}}, y^{(q-r-2)}] - [y,y_{i_{l-1}}, y^{(q-r-2)}][f, y^{(r)}],y_{i_{l+q-1}},\ldots,y_{i_{n}}]\\
                &= [[f,y^{(q-1)}]y_{i_{l-1}} - y_{i_{l-1}}[f,y^{(q-1)}],y_{i_{l+q-1}},\ldots,y_{i_{n}}] + d
            \end{split}
        \end{equation*}
        where $d$ has the required form. Write $g = [f,y^{(q-1)}] = [y, y_{i_{1}},\ldots,y_{i_{l-2}},y^{(q-1)}]$. By induction hypothesis, $g$ is linear combination of products of L-normal polynomials such that the last variable is $y_{i_{l-2}}$, modulo $I(\varepsilon)$. Since $i_{l-2} \leq i_{l-1}$, the lemma holds valid for the the polynomial $[g y_{i_{l-1}} - y_{i_{l-1}} g, y_{i_{l+q-1}},\ldots,y_{i_{n}}]$, which equals $[g, y_{y_{i_{l-1}}}, y_{i_{l+q-1}},\ldots,y_{i_{n}}]$. Therefore, it holds for $c$ as well.
    \end{proof}

    \begin{theorem}
    \label{mainTheorem}
     Let $\mathbb{F}$ be an arbitrary field (finite or infinite), $G$ an arbitrary group, and let $\varepsilon:\mathbb{N}\to G$ define a $G$-grading on $\mathrm{UT}$ via $\deg_{G} e_{i,i+1}=\varepsilon(i)$. Then:
        \begin{enumerate}
        \renewcommand{\labelenumi}{(\roman{enumi})}
            \item $\mathrm{Id}_{G}(\mathrm{UT}, \varepsilon)$ is generated by all polynomials of the kind $f_\eta$, where $\eta$ is an $\varepsilon$-bad sequence (see \Cref{badpolynomial}).
            \item A vector space basis of the relatively free $G$-graded algebra of $(\mathrm{UT},\varepsilon)$ consists of the polynomials
            \begin{equation}\label{basis}
            f = y_{1}^{a_{1}}\cdots y_{s}^{a_{s}} c_{1}\cdots c_{m}
            \end{equation}
            where $0 \leq  a_{1},\ldots,a_{s} < |\F|$, $(\deg_{G} c_{1},\ldots,\deg_{G} c_{m})$ is $\varepsilon$-good, and each $c_{i}$ is:
            \begin{enumerate}
                \item a semistandard commutator, if $\mathbb{F}$ is infinite (\Cref{def:semistandard}),
                \item an L-normal polynomial, if $\mathbb{F}$ is finite with $q$ elements (\Cref{Lnormal}).
            \end{enumerate}
        \end{enumerate}
    \end{theorem}
    \begin{proof}
        By \Cref{infiniteFieldBasis} and \Cref{spanSetFiniteField} and since $I(\varepsilon) \subseteq Id_{G}(UT, \varepsilon)$ we need to prove that the set of polynomials of the kind \eqref{basis} is linearly independent, modulo $\mathrm{Id}_G(\mathrm{UT},\varepsilon)$.

        Let $\sum \lambda  y_{1}^{a_{1}}\cdots y_{s}^{a_{s}} c_{1}\cdots c_{m}\in\mathrm{Id}_G(\mathrm{UT},\varepsilon)$, where $\lambda \in \mathbb{F}$ and $f = y_{1}^{a_{1}}\cdots y_{s}^{a_{s}} c_{1}\cdots c_{m}$ is as in \eqref{basis}. Clearly each sequence $(\deg_{G} c_{1},\ldots,\deg_{G} c_{m})$ is $\varepsilon$-good. By \Cref{propositionBadGood}, $\mathcal{G}(\varepsilon)= \bigcup\limits_{n\in\mathbb{N}}\mathcal{G}(\varepsilon_n)$, so there exist an $r \in \mathbb{N}$ such that  $(\deg_{G} c_{1},\ldots,\deg_{G} c_{m})$ is $\varepsilon_{r}$-good for each summand. By \Cref{lem1}, the map
        $$\varphi:\mathbb{F}\langle X^{G}\rangle/\mathrm{Id}_{G}(\mathrm{UT}, \varepsilon) \to \mathbb{F}\langle X^{G}\rangle/\mathrm{Id}_{G}(\mathrm{UT_{r}}, \varepsilon_{r})$$
        is a well-defined surjective homomorphism of algebras. From the basis of the relatively free algebra of $(\mathrm{UT}_r,\varepsilon_r)$ (see \cite[Theorem 2.8]{VinKoVa2004} for infinite field, and \cite[Theorem 3.21]{GR2020} for finite field), we have that each $\lambda$ has to be 0. Therefore, the set of polynomials of the kind \eqref{basis} is linearly independent, modulo $\mathrm{Id}_G(\mathrm{UT},\varepsilon)$.
    \end{proof}
    \begin{remark}
    If $\varepsilon$ is any $G$-grading on $\mathrm{UT}$, then it is known that the codimension sequence satisfies $c_m(\mathrm{UT})\le c_m^G(\mathrm{UT},\varepsilon)$ (see, for instance, \cite{BGR1998}). Thus, $c_m^G(\mathrm{UT},\varepsilon)\ge m!$ for all $m\in\mathbb{N}$. This happens since $\mathrm{UT}$ is not a PI-algebra.
    \end{remark}

    Note that if $G = \{1\}$ then any finite sequence of elements on $G$ is $\varepsilon$-good and we conclude the well-known result that $\mathrm{UT}$ has no polynomial identities. Furthermore, we obtain the following corollary.

    \begin{corollary}
        Let $\mathbb{F}$ be an arbitrary field. Then, a vector space basis of the free associative algebra $\mathbb{F}\langle X\rangle$ consists of the polynomials
        $$f = x_{1}^{a_{1}}\cdots x_{s}^{a_{s}} c_{1}\cdots c_{m}$$
        where $0 \leq  a_{1},\ldots,a_{s} < |\F|$ and each $c_{i}$ is:
        \begin{enumerate}
            \item a standard commutator with weight $1$, if $\mathbb{F}$ is infinite,
            \item of the form
            \begin{enumerate} 
                \item $c_{i} = [x_{i},x_{1}^{(d_{1})},\ldots,x_{i}^{(d_{i}-1)},\ldots,x_{s}^{(d_{s})}]$, $0 \leq d_{1},\ldots,d_{s} < q$, $s \geq 2$ or
                \item $c_{i} = [{x_{i}}^{q} - x_{i},x_{1}^{(d_{1})},\ldots,x_{i-1}^{(d_{i-1})},x_{i+1}^{(d_{i+1})},\ldots,x_{s}^{(d_{s})}]$, $0 \leq d_{1},\ldots,d_{s} < q$,
            \end{enumerate}
            if $\mathbb{F}$ is finite with $q$ elements. \qed
        \end{enumerate}
    \end{corollary}

\section{Conditions for the existence of graded identities}\label{sec:conditionsID} In this section, we shall find conditions so that a grading on $\mathrm{UT}$ satisfies some graded polynomial identity. As a consequence of our main theorem, we have that $\mathrm{Id}_G(\mathrm{UT},\varepsilon) = 0$ if and only if every finite sequence of elements of $G$ is $\varepsilon$-good. The next proposition simplifies this condition.

\begin{lemma}
    Let $G$ be a group and let $\varepsilon:\mathbb{N}\to G$ define a $G$-grading on $\mathrm{UT}$, and let $A$ be a set of generators of $G$ as a monoid. Then, $\mathrm{Id}_{G}(\mathrm{UT},\varepsilon)=0$ if, and only if, every finite sequence of elements of $A$ is $\varepsilon$-good.
\end{lemma}
    \begin{proof}
        It follows from \Cref{monoidGenerator} and \Cref{mainTheorem}.
    \end{proof}

   If $G$ is a finite group, then a set generates $G$ as a group if and only if it generates $G$ as a monoid. Thus, we obtain:

    \begin{Prop}
        Let $G$ be a finite group and let $\varepsilon:\mathbb{N}\to G$ define a $G$-grading on $\mathrm{UT}$. If $A$ is a set of generators of G (as a group), then $\mathrm{Id}_{G}(\mathrm{UT},\varepsilon)=0$ if and only if every finite sequence of elements of $A$ is $\varepsilon$-good. \qed
    \end{Prop}

     For each $i \in \mathbb{N}$ let $\rho_{i} = (d_{1}, d_{2}, \ldots)$, where $d_j=\deg_{G} e_{i,i+j}$. The following lemma emerges as a direct consequence of the definition of $\varepsilon$-good sequence and of the fact that $e_{ij}e_{jk} = e_{ik}$.

    \begin{lemma}
        \label{goodSequenceLineSequence}
        A finite sequence $\eta = (\eta_{1}, \ldots, \eta_{s})$ of elements on $G$ is $\varepsilon$-good if and only if there exists $i \in \mathbb{N}$ such that 
        $(\eta_{1}, \eta_{1}\eta_{2}, \ldots, \eta_{1}\cdots\eta_{s})$ is a subsequence (not necessarily consecutive) of $\rho_{i}$. \qed
    \end{lemma}

    \begin{remark}
        As a consequence of the previous lemma, we have that $\mathrm{Id}_{G}(\mathrm{UT},\varepsilon)=0$ if and only if for each finite sequence $\eta = (g_{1}, \ldots, g_{s})$ of elements of $G$, there exists $i \in \mathbb{N}$ such that $\eta$ is a subsequence (not necessarily consecutive) of $\rho_{i}$.
    \end{remark}

Given a $G$-grading on $\mathrm{UT}$ where each matrix unit is homogeneous, the subspace $R_i:=e_{ii}J(\mathrm{UT})$ is always graded. Now, we consider the following properties.

\begin{definition}\label{incomplete}
Let $\varepsilon:\mathbb{N}\to G$ define a $G$-grading on $\mathrm{UT}$. We say that:
\begin{enumerate}
\renewcommand{\labelenumi}{(\roman{enumi})}
\item $\varepsilon$ is \emph{eventually incomplete} if there exists $N\in\mathbb{N}$ such that $\mathrm{Supp}\,R_i\ne G$, for all $i \geq N$, where $R_i=e_{ii}J(\mathrm{UT})$,
\item the $m$-th row of $\varepsilon$ is \emph{$g$-finite}, where $g\in G$, if $\{n\in\mathbb{N}\mid n > m, \deg_{G} e_{mn}=g\}$ is finite,
\item $\varepsilon$ is said to be \emph{row-incomplete} if it has a $g$-finite row, for some $g\in G$,
\item $\varepsilon$ is said to be \emph{row-finite} if for each $m \in \mathbb{N}$, there exists $g \in G$ such that the $m$-th row is $g$-finite,
\item the $m$-th row of $\varepsilon$ is \emph{totally complete} if for any $N\in\mathbb{N}$ and $g\in G$, there exists $\ell>N$ such that $\deg_{G} e_{m\ell}=g$.
\end{enumerate}
\end{definition}

All the previous definitions turn to be equivalent in the following sense:
\begin{lemma}\label{equivalences_incompelte}
Let $\varepsilon:\mathbb{N}\to g$ define a $G$-grading on $\mathrm{UT}$. The following conditions are equivalent:
\begin{enumerate}
\renewcommand{\labelenumi}{(\roman{enumi})}
\item $\varepsilon$ is eventually incomplete,
\item there exists $N\in\mathbb{N}$ such that no $m$-th row of $\varepsilon$ is totally complete, for $m \geq N$,
\item $\varepsilon$ is row-incomplete,
\item $\varepsilon$ is row-finite.
\end{enumerate}
\end{lemma}
\begin{proof} 
    The chain of implications $(i)\Rightarrow(ii)\Rightarrow(iii)$ is clear. Assume that $\varepsilon$ has a $g$-finite row $m$, for some $g \in G$, and denote $\varepsilon=(g_{1},g_{2},\ldots)$. Consider the sequence $\varepsilon' = (g_{m},g_{m}g_{m+1},\ldots)$ of the $G$-degrees of the matrix units at the $m$-th row. Since $g$ appears a finite number of times in $\varepsilon'$, for each $m' > m$, the element $g_{m'}^{-1}\cdots g_{m}^{-1}g$ appears a finite number of times in the $m'$-th row. Similarly, for each $m' < m$, the element $g_{m'}\cdots g_{m}g$ appears in the $m'$-th row a finite number of times.

    Finally, we need to prove that $(iv)\Rightarrow(i)$. For it, assume that $\varepsilon$ is row-finite, and let $g\in G$ be such that $\{n\in\mathbb{N}\mid n > 1, \deg_{G} e_{1n}=g\}$ is finite. If this set is empty we are done. Suppose that it is not empty and let $N$ be the greatest element in this set. It means that $g$ appears only a finite number of times in the sequence $\varepsilon' = (g_{1},g_{1}g_{2},\ldots)$, and the last occurence is at the postion $N$. So, the element $g' = g_{N}^{-1}\cdots g_{1}^{-1}g$ does not appear in the sequence $(g_{N+1}, g_{N+1}g_{N+2},\ldots)$. Therefore, $g'$ does not appears in $\mathrm{Supp}(R_{N+1})$, and $\mathrm{Supp}(R_{N+1}) \neq G$.
\end{proof}

     Recall that given a map $\bar{g}:\mathbb{N}\to G$ we have a $G$-grading on $\mathrm{UT}$ by imposing $\deg_{G} e_{ij}=\bar{g}(i)\bar{g}(j)^{-1}$. The previous definition can be stated in terms of such map $\bar{g}$ in the following way.

     \begin{lemma}
         Let $\bar{g}:\mathbb{N}\to G$ be a map and define a $G$-grading on $\mathrm{UT}$ by setting $\deg_{G}(e_{ij}) = \bar{g}(i)\bar{g}(j)^{-1}$. The grading is eventually incomplete if and only if there exist $N \in \mathbb{N}$ and $g \in G$ such that $\bar{g}(i) \neq g$, for all $i > N$.
     \end{lemma}
     \begin{proof}
         Suppose that this $G$-grading is eventually incomplete and let $N \in \mathbb{N}$ such that $\mathrm{Supp}(R_{i}) \neq G$, for all $i \geq N$. Then, there is some $g' \in G$ such that $g' \notin \mathrm{Supp}(R_{N})$. Therefore, setting $g = (g')^{-1}\bar{g}(N)$ we have that $\bar{g}(i) \neq g$, for all $i > N$.
         Now, suppose that there exist $N \in \mathbb{N}$ and $g \in G$ such that $\bar{g}(i) \neq g$, for all $i > N$. Then, for all $i \geq N$, $\bar{g}(i)g^{-1} \notin \mathrm{Supp}(R_{i})$.
     \end{proof}

     We may translate the map $\bar{g}$ in such a way that the above element $g$ equals $1$. In other words, assume that $\bar{g}:\mathbb{N}\to G$ is a map such that there exist $N \in \mathbb{N}$ and $g \in G$ satisfying $\bar{g}(i) \neq g$, for all $i > N$. Define $\bar{f}:\mathbb{N}\to G$ by $\bar{f}(i) = \bar{g}(i)g^{-1}$. Note that the $G$-gradings obtained by setting, respectively, $\deg_{G}(e_{ij})=\bar{g}(i)\bar{g}(j)^{-1}$ and $\deg_{G}(e_{ij})=\bar{f}(i)\bar{f}(j)^{-1}$ are the same. Furthermore, $\bar{f}(i) \neq 1$ for all $i > N$. Thus, an elementary $G$-grading is eventually incomplete if and only if there exist a map $\bar{f}:\mathbb{N}\to G$ and a natural $N \in \mathbb{N}$ such that $\deg_{G}(e_{ij}) = \bar{f}(i)\bar{f}(j)^{-1}$ and $\bar{f}(i) \neq 1$ for all $i > N$. 
     
\begin{Prop}
    Let $\bar{g}:\mathbb{N}\to G$ be a map such that there exists $N \in \mathbb{N}$ satisfying $\bar{g}(i) \neq 1$, for all $i > N$. Define a $G$-grading on $\mathrm{UT}$ by setting $\deg_{G}(e_{ij}) = \bar{g}(i)\bar{g}(j)^{-1}$. Then, $\bar{g}(i) \notin R_{i}$ for all $i \geq N$. In particular, if $\bar{g}(N) = 1$ then $1 \notin R_{N}$.
\end{Prop}
\begin{proof}
   Since $\bar{g}(j) \neq 1$ for all $j > N$, then $\deg_{G}(e_{ij}) = \bar{g}(i)\bar{g}(j)^{-1} \neq \bar{g}(i)$ for all $j > i \geq N$. Therefore, $\bar{g}(i) \notin R_{i}$ for all $i \geq N$.
\end{proof}
     
Thus, we obtain a criterion to the non-existence of graded polynomial identities.
\begin{Prop}
    \label{notEventuallyIncomplete}
    If $\varepsilon$ is not eventually incomplete, then $\mathrm{Id}_G(\mathrm{UT},\varepsilon) = 0$.
\end{Prop}
\begin{proof}
    From \Cref{equivalences_incompelte}.(iv), $\varepsilon$ is not row-finite. Thus, for every element $g \in G$ and $n \in \mathbb{N}$, there exists $m \in \mathbb{N}$, $m > n$, such that $\deg_{G}(e_{nm}) = g$. Hence, every finite sequence of elements from $G$ is  $\varepsilon$-good.
\end{proof}

The converse of this proposition holds for a finite group. To prove it we will need some technical results.

\begin{definition}
    We say that an elementary grading defined by $\varepsilon = (g_{1},g_{2},\ldots)$ \emph{stabilizes at the $N$-th row} if $\mathrm{Supp}(R_{i+1}) = g_{i}^{-1}\mathrm{Supp}(R_{i})$, for all $i \geq N$. In this case, we say that $\varepsilon$ is \emph{$N$-stable} and, if $N = 1$, we say that $\varepsilon$ is \emph{stable}.
\end{definition}

For instance, note that if the first row of a grading $\varepsilon$ is totally complete, then every row of $\varepsilon$ is totally complete as well. Thus, $\varepsilon$ is stable.

\begin{lemma}
    \label{not1}
    If $s \in \mathbb{N}$, $s > 1$, then $1 \in \mathrm{Supp}(R_{s})$ if and only if $\mathrm{Supp}(R_{s}) = g_{s-1}^{-1}\mathrm{Supp}(R_{s-1})$.
\end{lemma}
    \begin{proof}
        Note that $\mathrm{Supp}(R_{s-1}) = (g_{s-1}\mathrm{Supp}(R_{s})) \cup \{g_{s-1}\}$. Thus, $g_{s-1}^{-1}\mathrm{Supp}(R_{s-1}) = \mathrm{Supp}(R_{s}) \cup \{1\}$. Hence, $1 \in \mathrm{Supp}(R_{s})$ if and only if $\mathrm{Supp}(R_{s}) = g_{s-1}^{-1}\mathrm{Supp}(R_{s-1})$.
    \end{proof}

As a consequence, we obtain a criteria to determine if a grading is not stable.
\begin{corollary}\label{finite_indices}
    If $G$ is a group and $\varepsilon$ is $N$-stable, for some $N \in \mathbb{N}$, then there exist a finite number of indices $i_{1},\ldots,i_{s} \in \mathbb{N}$ such that $1 \notin \mathrm{Supp}\,R_{i}$, for $i = i_{1}, \ldots, i_{s}$. \qed
\end{corollary}

We determine an equivalent condition for stability.
\begin{lemma}
    \label{emptyInfinite}
    Let $N \in \mathbb{N}$ and $\varepsilon = (g_{1},g_{2},\ldots)$ define an elementary grading on $\mathrm{UT}$. Then, $\varepsilon$ is $N$-stable if, and only if, for all $g \in G$ and $m \geq N$, $\{n\in\mathbb{N}\mid n > m, \deg_{G}(e_{mn})=g\}$ is either empty or infinite.
\end{lemma}
\begin{proof}
    Suppose that there exist $m' \in \mathbb{N}$, $m' \geq N$, and $g \in G$ such that $\{n\in\mathbb{N}\mid n > m', \deg_{G}(e_{m'n})=g\} \neq \emptyset$ is finite. Let $n'$ be the greatest element of this set. 
    If $1\in\mathrm{Supp}\,R_{n'}$, then we can find $\ell>n'$ such that $\deg_{G} e_{n'\ell}=1$. Thus, $\deg_{G} e_{m'\ell}=\deg_{G} e_{m'n'}\deg_{G} e_{n'\ell}=g$, a contradiction with the choice of $n'$. Therefore, \Cref{not1} implies that $\mathrm{Supp}(R_{n'}) \neq g_{n'-1}^{-1}\mathrm{Supp}(R_{n'-1})$.
    
    Conversely, suppose that $\mathrm{Supp}(R_{m'+1}) \neq g_{m'}^{-1}\mathrm{Supp}(R_{m'})$, for some $m' \geq N$. By \Cref{not1}, $1\notin\mathrm{Supp}(R_{m'+1})$. Now, assume that, for some $l>m'+1$, one has $g_m'=\deg_{G} e_{m'\ell}=\deg_{G} e_{m',m'+1} \deg_{G} e_{m'+1,\ell}$. Since $\deg_{G} e_{m',m'+1}=g_m'$, then $\deg_{G} e_{m'+1,\ell}=1$, a contradiction. Hence, $|\{n\in\mathbb{N}\mid n > m', \deg_{G} e_{m'n}=g_{m'}\}| = 1$.
\end{proof}

    \begin{remark}
        There exist gradings on $\mathrm{UT}$ that do not stabilize. Consider the $\mathbb{Z}$-grading on $\mathrm{UT}$ defined by $\varepsilon=(0,-1,2,-3,\ldots)$. The sequence of the $\mathbb{Z}$-degrees of the (strictly upper) matrix units at the first row is $\rho_{1} = (0,-1,1,-2,2,\ldots)$. In other words, each element of $\mathbb{Z}$ appears once in $\rho_{1}$. It implies that, for all $m \in \mathbb{N}$, there exists an element $g \in \mathbb{Z}$ such that $\{n\in\mathbb{N}\mid n > m, \deg_{G} e_{mn}=g\}$ is a singleton.
    \end{remark}
If $G$ is a finite group, then we can guarantee that every grading on $\mathrm{UT}$ stabilizes, as we shall prove as follows. This will be important to establish the next main results.
\begin{lemma}
    \label{finiteGroupStabilizes}
        If $G$ is a finite group, then every elementary $G$-grading on $\mathrm{UT}$ stabilizes.
\end{lemma}
\begin{proof}
        Let $\varepsilon = (g_{1}, g_{2}, \ldots)$ define an elementary $G$-grading on $\mathrm{UT}$. If each $g\in\mathrm{Supp}\,R_1$ appears as the degree of infinitely many matrix units in the first row, then the same is true for every row of $\varepsilon$. Thus, \Cref{emptyInfinite} implies that $\varepsilon$ is stable. Otherwise, let
        \[
        R=\{g\in G\mid\text{$g\in\mathrm{Supp}\,R_1$ and $\exists m_g\in\mathbb{N}$ s.t.~$\deg_Ge_{1\ell}\ne g=\deg_Ge_{1m_g}$, $\forall l\ge m_g$}\},
        \]
        and assume that $R\ne\emptyset$. Let $m=\max\{m_g\mid g\in R\}$.

        For all $g\in R$ and $m' \geq m + 1$, the element $(g_{1}\cdots g_{m'-1})^{-1}g$ does not belong to $\mathrm{Supp}\,R_{m'}$. Indeed, otherwise, assume that $\deg_{G} e_{m'\ell}=(g_{1}\cdots g_{m'-1})^{-1}g$. Then, $\deg_{G} e_{1\ell}=\deg_{G} e_{1m'}\deg_{G} e_{m'\ell}=g$. Since $\ell>m$, we obtain a contradiction with the choice of $m$. Conversely, assume that $h\in\mathrm{Supp}\,R_{m'}$ appears only a finite number of times as a degree of some strict upper matrix unit in $R_{m'}$. Then, a similar argument shows that $g_1\cdots g_{m'}h\in R$, a contradiction. Therefore, from \Cref{emptyInfinite}, $\varepsilon$ is $(m + 1)$-stable.
    \end{proof}

    The connection between stability and good sequences are given in the following lemmas.
    \begin{lemma}
        \label{goodSequenceLineStable}
        Let $\varepsilon$ be a stable $G$-grading on $\mathrm{UT}$ and $\eta = (\eta_{1}, \ldots, \eta_{m})\in G^m$. Then, $\eta$ is $\varepsilon$-good if and only if there exists $i \in \mathbb{N}$ such that 
        $$\eta_{1}, \eta_{1}\eta_{2}, \ldots, \eta_{1}\cdots\eta_{m} \in \mathrm{Supp}(R_{i}).$$
    \end{lemma}
    \begin{proof}
        It follows from \Cref{emptyInfinite} and \Cref{goodSequenceLineSequence}.
    \end{proof}

    Replacing stability by $N$-stability, we obtain a version of the previous lemma.
    \begin{lemma}
        \label{goodSequenceLine}
        Let $\varepsilon$ be an $N$-stable $G$-grading on $\mathrm{UT}$, for some $N \in \mathbb{N}$, and $\eta = (\eta_{1}, \ldots, \eta_{N}, \ldots, \eta_{m})$ be $\varepsilon$-good, where $m>N$. Then, there exists $i \geq N$ such that $\eta_{N}$, $\eta_{N}\eta_{N+1}$, \dots, $\eta_{N}\cdots\eta_{m} \in \mathrm{Supp}(R_{i})$.
    \end{lemma}
    \begin{proof}
        Write $\varepsilon = (g_{1}, g_{2}, \ldots)$. The sequence $\varepsilon' = (g_{N}, g_{N+1}, \ldots)$ defines a stable $G$-grading on $\mathrm{UT}$. Moreover, $(\mathrm{UT}, \varepsilon')$ is a graded subalgebra of $(\mathrm{UT}, \varepsilon)$. Thus, if $\eta = (\eta_{1}, \ldots, \eta_{N}, \ldots, \eta_{m})$ is an $\varepsilon$-good sequence, then $\eta' = (\eta_{N}, \ldots, \eta_{m})$ is an $\varepsilon'$-good sequence. By \Cref{goodSequenceLineStable}, there exists $i \geq N$ such that $\eta_{N}, \eta_{N}\eta_{N+1}, \ldots \eta_{N}\cdots\eta_{m} \in \mathrm{Supp}(R_{i})$, where $R_i=e_{ii}J(\mathrm{UT})$ has the grading induced by $(\mathrm{UT}, \varepsilon)$.
    \end{proof}

    Now, we can finally state.

    \begin{theorem}\label{conditionID}
        Let $\mathbb{F}$ be an arbitrary field (finite or infinite), $G$ a finite group, and $\varepsilon:\mathbb{N}\to G$ define a $G$-grading on $\mathrm{UT}$. Then $\varepsilon$ is eventually incomplete (see \Cref{incomplete}) if and only if $\mathrm{Id}_{G}(\mathrm{UT},\varepsilon) \ne 0$.
    \end{theorem}
    \begin{proof}
        Let $G = \{g_{1}, \ldots, g_{m}\}$ be a finite group, $\varepsilon:\mathbb{N}\to G$ define a $G$-grading on $\mathrm{UT}$, and assume that $\mathrm{Id}_{G}(\mathrm{UT},\varepsilon) = 0$. Let $n\in\mathbb{N}$. We shall prove that $\varepsilon$ is not eventually incomplete by showing that there exists $i>n$ such that $\mathrm{Supp}\,R_{i}=G$. By \Cref{finiteGroupStabilizes}, $\varepsilon$ is $N$-stable for some $N \in \mathbb{N}$. By definition, possibly replacing $N$ by a greater integer, we can assume that $N>n$. Since $(\mathrm{UT},\varepsilon)$ satisfies no graded polynomial identity, the sequence $$
(\underbrace{g_{1},\ldots,g_1}_\text{$N$ times}, g_{1}^{-1}g_2, g_2^{-1}g_3, \ldots, g_{m-1}^{-1}g_{m})
        $$
        is an $\varepsilon$-good sequence. Thus, by \Cref{goodSequenceLine}, there exists $i \geq N$ such that $g_{1}$, \dots, $g_{m} \in \mathrm{Supp}(R_{i})$. It means that $\mathrm{Supp}(R_{i}) = G$. By consequence, $\varepsilon$ is not eventually incomplete. The converse is proved in \Cref{notEventuallyIncomplete}.
    \end{proof}

    \begin{remark}
        The theorem does not hold if the group $G$ is infinite. In fact, let $\mathcal{N} = \{\mathcal{N}_{i}: i \in \mathbb{N}\}$ be a partition of $\mathbb{N}$ where each $\mathcal{N}_{i}$ is infinite. Then, for each $i \in \mathbb{N}$, let $\varphi_{i}:\mathcal{N}_{i}\to \mathbb{Z}\setminus\{0\}$ be a bijection. Define the function $\bar{f}:\mathbb{N}\to \mathbb{Z}$ via $x \mapsto \varphi_{i}(x)$ if $x \in \mathcal{N}_{i}$, and define a $\mathbb{Z}$-grading $\varepsilon$ on $\mathrm{UT}$ via $\deg_{\mathbb{Z}}(e_{ij}) = \bar{f}(i)-\bar{f}(j)$. For each $t \in \mathbb{N}$, there exists $z \in \mathbb{Z}$ such that $\mathrm{Supp}(R_{t}) = \{x \in \mathbb{Z}: x \neq z\}$, so $\varepsilon$ is eventually incomplete. Note that $\varepsilon$ is stable. Moreover, by the definition of $\bar{f}$, for each finite sequence $(z_{1}, \ldots, z_{s})$ of elements of $\mathbb{Z}$, there exists $t \in \mathbb{N}$ such that $z_{1}, \ldots, z_{s} \in \mathrm{Supp}(R_{t})$. Thus, every finite sequence is $\varepsilon$-good. By \Cref{goodSequenceLineStable} and \Cref{mainTheorem}, $\mathrm{Id}_{G}(\mathrm{UT},\varepsilon) = 0$.
    \end{remark}

In what follows, we build an example of a $G$-grading $\varepsilon$ where $G$ is infinite, $\varepsilon$ is stable, each element of $G$ is realized as a homogeneous degree of a matrix unit in infinitely many distinct rows, and $\mathrm{Id}_{G}(\mathrm{UT},\varepsilon) \neq 0$. Let $\mathcal{N} = \{\mathcal{N}_{i}: i \in \mathbb{N}\}$ be a partition of $\mathbb{N}$, where each $\mathcal{N}_{i}$ is infinite. Then, for each $i \in \mathbb{N}$, let $\varphi_{i}:\mathcal{N}_{i}\to (2\mathbb{Z}+1)\cup\{2\}$ be a bijection. Define the map $\bar{f}:\mathbb{N}\to \mathbb{Z}$ via $x \mapsto \varphi_{i}(x)$ if $x \in \mathcal{N}_{i}$, and define a $\mathbb{Z}$-grading $\varepsilon$ on $\mathrm{UT}$ via $\deg_{\mathbb{Z}}(e_{ij}) = \bar{f}(i)-\bar{f}(j)$. Note that, for each $z \in \mathbb{Z}$, there exist infinitely many $t \in \mathbb{N}$ such that $z \in \mathrm{Supp}(R_{t})$. Besides, the support of each $R_{t}$ is either $(2\mathbb{Z} + 1)\cup \{0\}$ or of the form $2\mathbb{Z}\cup\{2k+ 1\}$, for some $k \in \mathbb{Z}$. In particular, $(1,2,3)$ is an $\varepsilon$-bad sequence.

\section{Finite basis property}\label{sec:finitebasis}
In this section, we prove that, if $G$ is a finite group, then $\mathrm{Id}_{G}(\mathrm{UT})$ is finitely based.

For the following, let $G$ be an arbitrary group and let $\varepsilon: \mathbb{N}\to G$ define a $G$-grading on $\mathrm{UT}$ via $\deg_{G}(e_{i,i+1}) = \varepsilon(i)$.

\begin{definition}
An $\varepsilon$-bad sequence $\eta = (\eta_{1},\eta_{2},\ldots,\eta_{m})$ is called \emph{minimal} if
\begin{enumerate}
    \item every consecutive subsequence of $\eta$ is an $\varepsilon$-good sequence, and
    \item $(\eta_{1},\eta_{2},\ldots,\eta_{p-1},\eta_{p}\eta_{p+1},\eta_{p+2},\ldots,\eta_{m}) \in \mathcal{G}(\varepsilon)$, for all $1 \leq p < m$.
\end{enumerate}
\end{definition}

We obtain a version of \Cref{mainTheorem} in terms of minimal bad sequences.
\begin{Prop}
    \label{minimalBasis}
    $\mathrm{Id}_{G}(\mathrm{UT},\varepsilon)$ is generated, as an ideal of graded polynomial identities, by the set of all polynomials $f_{\eta}=f_{1}\cdots f_{m}$, where $\eta=(\deg_{G}(f_{1}),\ldots,\deg_{G}(f_{m}))$ is a minimal $\varepsilon$-bad sequence, and each $f_{i}$ has one of the following forms:
    \newcounter{contador}
    \begin{enumerate}
        \item $f_{i} = z_{2i}$, if $\deg_G f_i\ne1$,
        \item $f_{i} = z_{2i}z_{2i+1}$, if $\deg_G f_i=1$, where $\deg z_{2i}\ne1$,
        \item $f_{i} = [y_{2i},y_{2i+1}]$, if $\deg_G f_i=1$,
        \setcounter{contador}{\arabic{enumi}}
    \end{enumerate}
    Additionally, if $\mathbb{F}$ is a finite field containing $q$ elements, then $f_i$ may have the following form:
    \begin{enumerate}
    \setcounter{enumi}{\arabic{contador}}
        \item $f_{i} = y_{2i}^{q} - y_{2i}$, if $\deg_G f_i=1$.
    \end{enumerate} 
\end{Prop}
\begin{proof}
    By \Cref{mainTheorem}, we know that $\mathrm{Id}_{G}(\mathrm{UT},\varepsilon) = I(\varepsilon)$. Now, we consider the set of all polynomials $f_\eta=f_1\cdots f_m$, where each $f_i$ is as in (1)-(4), and $\eta$ is $\varepsilon$-bad. Since each polynomial of \Cref{badpolynomial} is included in such a list, we know that it generates $\mathrm{Id}_G(\mathrm{UT},\varepsilon)$.
    
    Let $\eta$ be an $\varepsilon$-bad sequence, not necessarily minimal, and let $f_{\eta} = f_{1}\cdots f_{m}$ be a polynomial as in the statement of the proposition. We shall prove that $f_\eta$ is a consequence of some $f_\mu$, as in the proposition, where $\mu$ is a minimal $\varepsilon$-bad sequence.
    
    Clearly, we may assume that every consecutive subsequence of $\eta$ is $\varepsilon$-good. If $\eta$ is minimal, then we are done. So, suppose that $\eta$ is not minimal. There exists $i \in \mathbb{N}$, such that 
    $$\eta' = (\deg_{G}(f_{1}),\ldots,\deg_{G}(f_{i-1}),\deg_{G}(f_{i})\deg_{G}(f_{i+1}),\deg_{G}(f_{i+2}),\ldots,\deg_{G}(f_{m}))$$
    is an $\varepsilon$-bad sequence.
    If $g=\deg_{G}(f_{i})\deg_{G}(f_{i+1})\ne1$, then $f_\eta\in\langle f_{\eta'}\rangle_G$, where $f_{\eta'}=f_1\ldots f_{i-1}x^{(g)}f_{i+2}\cdots f_m$. If $g=1$, then we may assume that $\deg_Gf_i\ne1$. Indeed, if $\deg_Gf_i=1$, then we may rewrite the sequence $\eta'$ as
    $$
    \eta'=(\deg_{G}(f_{1}),\ldots,\deg_Gf_{i-2},\deg_Gf_{i-1}\deg_{G}(f_{i}),\deg_{G}(f_{i+1}),\ldots,\deg_{G}(f_{m})).
    $$
    If $\deg_G f_{i-1}f_{i}\ne1$, then we are in the previous situation. Otherwise, we continue the process. Note that we cannot have $i=1$, otherwise $\eta'$ is a consecutive subsequence of $\eta$. Hence, the process eventually ends so we may assume that $\deg_Gf_i\ne1$. Now, write $f_{\eta'}=f_1\ldots f_{i-1}g_i f_{i+1}\cdots f_m$, where $g_i=x_{4i}^{(h)}x_{4i}^{(h^{-1})}$, $h=\deg_Gf_i$. Thus, $f_{\eta'}$ is as required and $f_\eta$ is a consequence of $f_{\eta'}$.
       
    If $\eta'$ is not minimal, we shall repeat the process. It ends because in each step the length of the sequence decreases and the empty sequence is $\varepsilon$-good.
\end{proof}

With these results we can state.

\begin{theorem}\label{thm:finitebasis}
    Let $\mathbb{F}$ be an arbitrary field (finite or infinite), $G$ a finite group, and let $\varepsilon:\mathbb{N}\to G$ define a $G$-grading on $\mathrm{UT}$. Then  $\mathrm{Id}_{G}(\mathrm{UT},\varepsilon)$ is finitely based.
\end{theorem}

\begin{proof}
    Assume that $G$ contains $m$ elements. By \Cref{finiteGroupStabilizes}, $\varepsilon$ is $N$-stable for some $N \in \mathbb{N}$. Assume that $\eta = (\eta_{1}, \ldots, \eta_{s})$ is a minimal $\varepsilon$-bad sequence with length $s > N + m - 1$. The sequence $(\eta_{N}, \eta_{N}\eta_{N+1}, \ldots, \eta_{N}\cdots\eta_{s})$ contains at least $m+1$ elements. Thus, there exist $l$, $k \in \mathbb{N}$, $N \leq k < \ell \leq s$, such that $\eta_{N}\cdots\eta_k = \eta_{N}\cdots\eta_\ell$. Define
    \[
    \eta' = (\eta_{1}, \ldots, \eta_{k - 1}, \eta_{k}\eta_{k+1}, \eta_{k + 2}, \ldots, \eta_{s}).
    \]
    Since $\eta$ is minimal, $\eta'$ is $\varepsilon$-good. Then, we can find 
    \[
    i_1<i_2<\cdots<i_N<\cdots<i_{k-1}<i_{k+1}<\cdots<i_{s+1}
    \]
    such that $\deg_G e_{i_q,i_{q+1}}=\eta_p$, for $q=1,2,\ldots,k-2,k+1,\ldots,s$, and $\deg_G e_{i_{k-1},i_{k+1}}=\eta_k\eta_{k+1}$. Thus, one has
    \[
    \deg_G e_{i_{N},i_{q+1}}=\eta_{N}\cdots\eta_q,\quad q=N,\ldots,k-1,k+1,\ldots,s.
    \]
    However, the 'missing element' is $\eta_{N}\cdots\eta_k=\eta_{N}\cdots\eta_\ell\in\mathrm{Supp}\,R_{i_{N}}$. Since $i_{N}\ge N$, from \Cref{emptyInfinite}, the elements $\eta_{N}\cdots\eta_q$, for $q=N,\ldots,s$, are the homogeneous degree of infinitely many matrix units in $R_{i_{N}}$. Thus, we can find $j_{N}<\cdots<j_s$ satisfying $\deg_G e_{i_{N},j_q}=\eta_{N}\cdots\eta_q$, for $q=N,\ldots,s$. In particular, $\deg_G e_{i_{N},j_{N}}=\eta_{N}$, and $\deg_G e_{j_q,j_{q+1}}=\eta_{q+1}$, for $q=N,\ldots,s-1$. Hence, $\eta$ is an $\varepsilon$-good sequence as well, a contradiction.

    As a consequence, every minimal $\varepsilon$-bad sequence has length at most $N+m-1$. Since $G$ is finite, it follows from \Cref{minimalBasis} that  $\mathrm{Id}_{G}(\mathrm{UT},\varepsilon)$ has a finite basis.
\end{proof}

\begin{remark}
Assume that a $G$-grading $\varepsilon$ is $N$-stable, where $G$ is finite. The proof of the previous theorem shows that the minimal bad sequences have length at most $N+|G|-1$. This bound is the best possible. Indeed, let $G=\mathbb{Z}_2$, and let $\varepsilon=(1,0,0,\ldots)$ define a $\mathbb{Z}_2$-grading on $\mathrm{UT}$ via $\deg_{\mathbb{Z}_2} e_{i,i+1}=\varepsilon(i)$. Then, by definition, $\varepsilon$ is stable (i.e., $1$-stable). The $\varepsilon$-bad sequences of $\varepsilon$ of length at most $2$ are $(0,1)$ and $(1,1)$. Thus, these sequences are minimal, and have length $2=N+|G|-1$.
\end{remark}

\begin{remark}
Let $G$ be an arbitrary group, $\Gamma$ a finite $G$-grading on $\mathrm{UT}$, and let $N=\langle x^{(g)}\mid g\notin\mathrm{Supp}\,\Gamma\rangle_G$ be the $T_G$-ideal generated by all the trivial $G$-graded polynomial identities of $\Gamma$. The proof of \Cref{finiteGroupStabilizes} works assuming that $\Gamma$ is finite. Thus, $\Gamma$ is $N$-stable, for some $N\in\mathbb{N}$. Moreover, a minor modification of the proof of \Cref{thm:finitebasis} shows that every minimal $\Gamma$-bad sequence has length at most $N+|\mathrm{Supp}\,\Gamma|-1$. Let $I$ be the $T_G$ ideal generated by all $f_\eta=f_1\cdots f_s$ (as in \Cref{minimalBasis}), where $\eta$ is a minimal $\Gamma$-bad sequence, and $\deg_G f_i\in\mathrm{Supp}\,\Gamma$. Then, $I$ is finitely based. Therefore, \Cref{minimalBasis} tells us that $\mathrm{Id}_G(\mathrm{UT},\Gamma)=I+N$.
\end{remark}

\section{Nonisomorphic gradings satisfying the same identities}\label{sec:examples}
In this section, in contrast to the finite-dimensional case, we provide examples of non-isomorphic gradings that satisfy the same set of graded polynomial identities.

\subsection{Example} Let $G=\mathbb{Z}_2=\{0,1\}$ be the group with $2$ elements. We let $\varepsilon=(0,1,0,1,\ldots)$ and $\varepsilon'=(1,0,1,0,\ldots)$. As before, $\varepsilon$ and $\varepsilon'$ define $\mathbb{Z}_{2}$-gradings on $\mathrm{UT}$ by assigning degrees on the matrix units $e_{i,i+1}$. Since $\varepsilon\ne\varepsilon'$, one has $(\mathrm{UT},\varepsilon)\not\cong(\mathrm{UT},\varepsilon')$. On the other hand, the linear map $\iota:\mathrm{UT}\to\mathrm{UT}$ such that $\iota(e_{ij})=e_{i+1,j+1}$ is an injective algebra homomorphism. Moreover, both $\iota:(\mathrm{UT},\varepsilon)\to(\mathrm{UT},\varepsilon')$ and $\iota:(\mathrm{UT},\varepsilon')\to(\mathrm{UT},\varepsilon)$ are graded algebra monomorphisms. Thus, $\mathrm{Id}_G(\mathrm{UT},\varepsilon)=\mathrm{Id}_G(\mathrm{UT},\varepsilon')$. From \Cref{goodsequencesbadsuite}, we see that every sequence of elements of $\mathbb{Z}_2$ is good. Hence, both algebras satisfy no nontrivial graded polynomial identities.
    
\subsection{Example} We provide an example of non-isomorphic gradings on $\mathrm{UT}$ where the set of graded polynomial identities coincide and is non-trivial. For it, let $G=\mathbb{Z}_{3}=\{0,1,2\}$ be the group with $3$ elements. We let $\varepsilon=(1,2,1,2,\ldots)$ and $\varepsilon'=(2,1,2,1,\ldots)$. As the last example, $\varepsilon$ and $\varepsilon'$ define $\mathbb{Z}_{3}$-gradings on $\mathrm{UT}$ by assigning degrees on the matrix units $e_{i,i+1}$ such that $(\mathrm{UT},\varepsilon)\not\cong(\mathrm{UT},\varepsilon')$ but $\mathrm{Id}_G(\mathrm{UT},\varepsilon)=\mathrm{Id}_G(\mathrm{UT},\varepsilon')$. Additionally, by \Cref{goodsequencesbadsuite}, we have that $(1,1)$ is an $\varepsilon$-bad sequence (and an $\varepsilon'$-bad sequence as well). Therefore, $\mathrm{Id}_G(\mathrm{UT},\varepsilon)=\mathrm{Id}_G(\mathrm{UT},\varepsilon')\ne0$.

\subsection{Example} Let $G$ be a group and let $\varepsilon = (g_{1}, g_{2}, \ldots)$ define a stable $G$-grading on $\mathrm{UT}$. Then, consider the $G$-grading defined by the sequence $\varepsilon' = (g_{1}, 1^{(d_{1})}, g_{2}, 1^{(d_{2})}, \ldots)$  on $\mathrm{UT}$, where each $d_{j} \in \mathbb{Z}_{\geq 0}$. By \Cref{goodsequencesbadsuite}, we have that $\mathcal{G}(\varepsilon) \subseteq \mathcal{G}(\varepsilon')$. On the other hand, let $\eta' = (\eta'_{1}, \ldots, \eta'_{s}) \in \mathcal{G}(\varepsilon')$. Then, by \Cref{goodSequenceLineStable}, there exists $i' \in \mathbb{N}$ such that 
    $\eta'_{1}, \eta'_{1}\eta'_{2}, \ldots, \eta'_{1}\cdots\eta'_{s} \in \mathrm{Supp}(R_{i'}, \varepsilon')$. Since we construct $\varepsilon'$ by adding $1$'s between two elements of $\varepsilon$, by \Cref{not1}, we have that for each $k \in \mathbb{N}$, $\mathrm{Supp}(R_{k}, \varepsilon) = \mathrm{Supp}(R_{k'}, \varepsilon')$, for every $k' \in \mathbb{N}$, $k - 1 + d_{1} + \cdots + d_{k - 2} < k' \leq k + d_{1} + \cdots + d_{k - 1}$. It means that there exists $i \in \mathbb{N}$ such that 
    $\eta'_{1}, \eta'_{1}\eta'_{2}, \ldots, \eta'_{1}\cdots\eta'_{s} \in \mathrm{Supp}(R_{i}, \varepsilon)$. Then, by \Cref{goodSequenceLineStable}, $\eta' \in \mathcal{G(\varepsilon)}$ and $\mathcal{G}(\varepsilon) = \mathcal{G}(\varepsilon')$. This implies that $\mathrm{Id}_G(\mathrm{UT},\varepsilon)=\mathrm{Id}_G(\mathrm{UT},\varepsilon')$.

\end{document}